\documentclass[12pt]{amsart}

\usepackage{graphicx,epsfig}
\usepackage[dvips]{psfrag}
\usepackage{enumerate}
\usepackage{amssymb}

\usepackage{graphicx}

\usepackage[active]{srcltx}

\newcommand{\occult}[1]{}

\newcommand{\eqdef}{\stackrel{\scriptscriptstyle\rm def}{=}}

\newcommand{\dist}{{\rm dist}}

\def\cC{{\mathcal C}}

\newtheorem{thm}{Theorem}[section]
\newtheorem{cor}[thm]{Corollary}
\newtheorem{lem}[thm]{Lemma}
\newtheorem{sch}[thm]{Scholium}
\newtheorem{prop}[thm]{Proposition}
\newtheorem{cla}[thm]{Claim}
\newtheorem{rem}[thm]{Remark}

\newtheorem{que}[thm]{Question}

\usepackage{epsfig}
\usepackage[latin1]{inputenc}

\title[Entropy-expansivity and partial hyperbolicity]{Entropy-expansiveness for partially hyperbolic diffeomorphisms.}
\author{L. J. D\'{\i}az, T. Fisher,  M. J. Pacifico, and J. L. Vieitez
 }

\address{L.~J.~D\'{i}az, Depto. Matem\'{a}tica, PUC-Rio, Marqu\^{e}s de S. Vicente 225,
22453-900 Rio de Janeiro RJ, Brazil}
\email{lodiaz@mat.puc-rio.br}
\address{T.~Fisher, Department of Mathematics, Brigham Young University, Provo, UT 84602}
\email{tfisher@math.byu.edu}
\address{M.~J.~Pacifico, Instituto de Matematica, Universidade Federal do Rio de Janeiro, C. P. 68.530, CEP 21.945-970, Rio de Janeiro, R. J. , Brazil}
\email{pacifico@im.ufrj.br}
\address{J.~L.~Vieitez, Instituto de Matematica, Facultad de Ingenieria, Universidad de la Republica, CC30, CP 11300, Montevideo, Uruguay}
\email{ jvieitez226@gmail.com}
\thanks{This paper was partially supported
CNPq, Faperj, and PRONEX-Dynamical Systems (Brazil). The authors
thank IM-UFRJ and PUC-Rio for the support and hospitality during their
visits while preparing this paper.}

\date{May 5, 2009}

\date{\today}

\begin{document}

\begin{abstract}
We show that diffeomorphisms with a dominated splitting of
the form $E^s\oplus E^c\oplus E^u$, where $E^c$ is a nonhyperbolic
central bundle that splits in a dominated way into
1-dimensional subbundles, are entropy-expansive.
In particular, they have a principal symbolic extension and equilibrium states.

\end{abstract}

\maketitle

\section{Introduction}

In dynamical systems one often considers the following
three main levels of structure: measure theoretic, topological,
and infinitesimal (properties of the derivative). Connections
between such different levels have always been of high interest.
For example,  uniform hyperbolicity, an infinitesimal property, 
implies a rich structure on the other two levels. This paper is
part of a program that studies how more general infinitesimal
properties (partial hyperbolicity and existence
of dominated splittings)
 force a certain
topological and measure-theoretic behavior for the underlying
dynamics. Here we will focus on a special type of partial
hyperbolicity that will ensure the system is entropy-expansive.

A  diffeomorphism $f$ is $\alpha$-expansive, $\alpha>0$, if
$\dist(f^n(x),f^n(y))\leq \alpha$ for all $n \in \mathbb{Z}$ implies $x=y$.
Uniform hyperbolicity implies $\alpha$-expansiveness
for some $\alpha>0$. 
One can relax this condition requiring
entropy-expansiveness.
This notion,  introduced by Bowen \cite{Bo},  is characterized by the fact
that, for every small $\alpha>0$ and every point $x\in M$, the intersection of the sets 
$f^{-n}(B(f^n(x),\alpha))$, $n\in \mathbb{Z}$, has zero  topological entropy.
Here $B(x,\alpha)$ is the ball centered at $x$
of radius $\alpha$.

Entropy-expansive maps are not
necessarily expansive, but have similar properties to expansive
maps
 in regards to topological and measure theoretic entropy.
For instance, entropy-expansive maps always have equilibrium
states, \cite{Ke}, and symbolic extensions preserving the entropy
structure (called principal extensions), \cite{BFF}.
For a broad discussion between these notions see \cite[Section
1]{DN}.

Several results  illustrate the interplay between smoothness and
en\-tro\-py-expansive-like properties. First, it follows
from~\cite{Bu} and~\cite{BFF} that $C^{\infty}$ diffeomorphisms
are asymptotically $h$-expansive. In \cite{DM}  it was shown that
every $C^2$ interval map has a symbolic extension. A similar
result for $C^2$ surface diffeomorphisms can be found in
\cite{Bur}. These results support the conjecture of Downarowicz
and Newhouse~\cite{DN} that every $C^2$ diffeomorphism has a
symbolic extension. However, this conjecture does not hold for
$C^1$ diffeomorphisms in any manifold of dimension three or
higher, \cite{As,DiFi}.

In this paper we adopt a different approach and study the relation
between ``hyperbolic-like properties'' and entropy-expansiveness.
Indeed uniformly hyperbolic
diffeomorphisms are entropy-expansive. There are
also some results available  for ``weakly hyperbolic'' systems.
For instance, in~\cite{PaVi} for a surface diffeomorphisms $f$ and
a compact $f$-invariant set $\Lambda$ with a dominated splitting
it is shown that the map $f$ restricted to $\Lambda$ is
entropy-expansive. See also further related results in
\cite{PaVi2}. Finally, in \cite{CY} it is shown that every
partially hyperbolic set with a one-dimensional center direction
is entropy-expansive.

Here we continue with the above investigations and
consider partially hyperbolic sets whose center bundle is higher
dimensional, but splits in a dominated way into one-dimensional
subbundles. We prove that such diffeomorphisms are
entropy-expansive:

\begin{thm}\label{t.mainresultado}
Let $f$ be a diffeomorphism
 and $\Lambda$ be a compact $f$-invariant set admitting a dominated
splitting $E^s\oplus E_1\oplus\cdots \oplus E_k\oplus E^u$,
where $E^s$ is uniformly
contracting, $E^u$ is uniformly expanding, and all $E_i$ are one-dimensional.
Then $f|_{\Lambda}$ is entropy-expansive.
\end{thm}

In the previous theorem we allow for the bundles 
$E^s$ and $E^u$ to possibly be empty.

Observe that
 the conditions in the theorem are also
necessary for entropy-expansiveness
at least for $C^1$ generic diffeomorphisms.
More precisely, in contrast with our result, $C^1$
generically diffeomorphisms having a  central (non-hyperbolic)
indecomposable bundle of dimension at least two are not entropy expansive,
 \cite{DiFi,As}.  In fact, the
proof of these results  follows the methods introduced
in~\cite{DN} relating  homoclinic tangencies
to non-existence  of symbolic extensions.
Roughly, the existence of
an indecomposable central
of dimension two (or higher) leads to the appearance of 
persistent homoclinic
tangencies which in turns prevent entropy-expansiveness.
We observe that  the
hypotheses of Theorem~\ref{t.mainresultado} prevents the creation of
homoclinic tangencies by perturbations, see for instance \cite{W}.

\medskip

Next we derive some consequences of Theorem~\ref{t.mainresultado}.
In~\cite{BFF} it is shown that every entropy-expansive diffeomorphism has a principal symbolic extension.
 We then have the next corollary.

\begin{cor}\label{l.pse}
If $\Lambda$ and $f$ are as in Theorem~\ref{t.mainresultado}, then  $f|_{\Lambda}$ has a principal symbolic extension
\end{cor}

Since every entropy-expansive diffeomorphism has an equilibrium state we have the next result.

\begin{cor}\label{l.eq}
For $\Lambda$ and $f$ as in Theorem~\ref{t.mainresultado}, if $\varphi\in C^0(\Lambda)$, then $f|_{\Lambda}$ has an
equilibrium state associated with $\varphi$.
\end{cor}

As domination is a key ingredient in our constructions we have the next natural question.

\begin{que}
Let $f$ be a diffeomorphism and $\Lambda$ be a compact $f$-invariant set with a
$Df$-invariant splitting (not necessarily dominated)
 $T_\Lambda M=E^s\oplus E_1\oplus\cdots \oplus E_k\oplus E^u$,
 with $E^s$ uniformly contracting, $E^u$ uniformly expanding, and  $E_1, \dots, E_k$ one-dimensional. Is $f$ entropy-expansive?
\end{que}

We note that as we were preparing this paper Liao, Viana, and Yang~\cite{GVY} announced that diffeomorphisms far from
homoclinic tangencies satisfy Shub's Entropy Conjecture, see~\cite{S}, and have a principal symbolic extension.
This conjecture relates the topological entropy to the spectral radius of the action induced by the system on
the homology (see previous partial results in
\cite{SX}).

\medskip

This paper is organized as follows.  In Section~\ref{s.background} we provide background, 
including the existence of fake foliations.  
In Section~\ref{s.proof} we prove Theorem~\ref{t.mainresultado}.

\section{Definitions and background}\label{s.background}
We now recall the main concepts in
this paper; namely, the notions of entropy-expansiveness and dominated splittings.

\subsection{Entropy and symbolic extensions}
In what follows
$(X,d)$ is a compact metric space and $f$ is a continuous self-map of $X$.
The $d_n$ metric on $X$ is defined as
$$d_n(x,y):=\max_{0\leq i\leq n-1}\dist (f^i(x), f^i(y))$$
and is equivalent to $d$ and defined for all $n\geq 0$.

For a set
$Y\subset X$, a set $A\subset Y$ is {\it $(n,\epsilon)$-spanning}
if for any $y\in Y$ there exists a point $x\in A$ where
$d_n(x,y)<\epsilon$. The minimum cardinality of the
$(n,\epsilon)$-spanning sets of $Y$ is denoted $r_n(Y,\epsilon)$.
We let
\begin{equation}\label{e.entropy}
\bar{r}(Y,\epsilon):=\limsup_{n\rightarrow\infty}\frac{1}{n}\log r_n(Y,\epsilon) \quad \textrm{and}
\quad \tilde{h}(f,Y):=\lim_{\epsilon\rightarrow 0}\bar{r}(Y,\epsilon).
\end{equation}
To see that the last limit exists see for instance~\cite{Mi}.  The {\emph{topological entropy}}
$h_{\mathrm{top}}(f)$
of $f$ is  $\tilde{h}(f, X)$.

Given $\epsilon>0$ let
$$
\Gamma^+_{\epsilon}(x) :=\bigcap_{n=0}^{\infty}f^{-n}(B_{\epsilon}(f^n(x)))
$$
and set
$$
h^*_f(\epsilon):=\sup_{x\in X}\tilde{h}(f,\Gamma^+_{\epsilon}(x)).
$$
The map $f$ is {\emph{entropy-expansive,}} or {\emph{$h$-expansive}} for short, if
there exists some $c>0$ such that $h^*_f(\epsilon)=0$ for all $\epsilon\in (0,c)$.

If $f$ is a homeomorphism, then we  define
$$
\Gamma_{\epsilon}(x) \colon
=\bigcap_{n\in\mathbb{Z}}f^{-n}(B_{\epsilon}(f^n(x)))\textrm{ and
} h^*_{f, \mathrm{homeo}}(\epsilon)\colon =\sup_{x\in
X}\tilde{h}(f,\Gamma_{\epsilon}(x)).
$$
If $X$ is a compact space
and $f$ is a homeomorphism, then $h^*_{f}(\epsilon)=h^*_{f,
\mathrm{homeo}}(\epsilon)$, \cite{Bo}.

For an $f$-invariant measure $\mu$ the {\it measure theoretic entropy of $f$}  measures
the exponential growth of orbits under $f$ that are ``relevant" to $\mu$ and is denoted $h_{\mu}(f)$,
see for instance~\cite{KH} for a precise definition.  The {\it variational principle} states that if
$X$ is a compact metric space and $f$ is continuous, then $h_{\mathrm{top}}(f)=\sup_{\mu\in \mathcal{M}(f)}h_{\mu}(f)$,
where $\mathcal{M}(f)$ is the space of all invariant Borel probability measures for $f$.

If $f$ is a
homeomorphism and $\varphi\in C^0(X)$,
then the {\it pressure of $f$ with respect to $\varphi$ and
$\mu\in \mathcal{M}(f)$} is
$$P_{\mu}(\varphi, f):=h_{\mu}(f) +\int \varphi d\mu.$$  The topological pressure of  $(X,f)$, denoted
$P(\varphi, f)$,  corresponds to a ``weighted" topological entropy,
see~\cite[p.~623]{KH}.    The {\emph{variational principle for pressure}}
states that if  $f$ is a
homeomorphism of $X$ and $\varphi\in C^0(X)$ then
$$P(\varphi, f):=\sup_{\mu\in\mathcal{M}(f)}P_{\mu}(\varphi, f).$$
A measure $\mu$ such that $P(\varphi, f)=P_{\mu}(\varphi, f)$
is called an {\it equilibrium state}.

A dynamical system $(X,f)$ has a {\it symbolic extension} if there
exists a subshift $(Y,\sigma)$ and a continuous surjective map
$\pi:Y\rightarrow X$ such that $\pi \circ\sigma=f\circ\pi$. The
system $(Y,\sigma)$ is called an {\it extension} of $(X,f)$ and
$(X,f)$ is called a {\it factor} of $(Y,\sigma)$. Note that the subshift need not be
of finite type and the factor map may be infinite-to-one.  A nice form
of a symbolic extension is a {\it principal extension}, that is, an
extension given by a factor map which preserves entropy for every
invariant measure, see~\cite{BD}.


%

\subsection{Dominated splittings}\label{ss.hyperboliclike}
Through the rest of the paper we assume that $M$ is a finite dimensional, smooth,
compact, and boundaryless Riemannian manifold and $f:M\to M$ is a $C^1$ diffeomorphism.
An $f$-invariant set $\Lambda$ (not necessarily closed) has a {\it dominated
splitting\/} if the tangent bundle $T_{\Lambda}M$ has a
$Df$-invariant splitting $E\oplus F$ such that
\begin{enumerate}
\item[(i)]
 the bundles $E$ and $F$ are both non-trivial,
 \item[(ii)] the fibers $E(x)$
and $F(x)$ have dimensions independent of $x\in \Lambda$, and
\item[(iii)] there
exist $C>0$ and  $0<\lambda <1$ such that
$$
\|Df^n|E(x)\|\cdot \|Df^{-n}|F(f^n(x))\|\leq C\lambda^n,
$$
 for all
$x\in \Lambda$ and  $n\geq 0$.
\end{enumerate}

More generally, a $Df$-invariant splitting
$$
T_\Lambda M=E_1\oplus\cdots\oplus E_k
$$
is dominated if for all $i=1,\dots,(k-1)$ the splitting $T_\Lambda
M=E_1^i\oplus E_{i+1}^k$ is dominated, where $E_j^\ell=E_j\oplus
\cdots \oplus E_\ell$, for $1\le j\le\ell\le k$.

Note that the above definitions imply the continuity of the splittings.

\medskip

We consider a diffeomorphism $f$
 and  a compact $f$-invariant set $\Lambda$  with a dominated
splitting $E^s\oplus E_1\oplus\cdots \oplus E_k\oplus E^u$ as in Theorem~\ref{t.mainresultado}.
 For $x\in\Lambda$ and $i\in\{1,..., k\}$ let us denote
\begin{equation}\label{e.pp}
\begin{array}{rlll}
 E^{cs,i}(x) :=E^s(x)\oplus E_1(x)\oplus \cdots\oplus E_i(x)\textrm{ and}\\
 E^{cu,i}(x) :=E_{i}(x)\oplus\cdots \oplus E_k(x)\oplus E^u(x).
  \end{array}
\end{equation}
We also let
$E^{cs,0}=E^s$ and $E^{cu,k+1}=E^u$ and write $s=\dim (E^s)$ and $u=\dim (E^u)$.

  By definition $ E^{cs,i}(x)\oplus  E^{cu,i+1}(x)$ is a dominated splitting for $\Lambda$
  and
$$\|Df^n|_{ E^{cs,i}(x)}\|\cdot \|Df^{-n}|_{E^{cu,i+1}(x)}\|\leq C\lambda^n$$ for some $C\geq 1$ and
$\lambda\in(0,1)$, and all $i\in\{0,...,k\}$, $x\in\Lambda$, and $n\geq 0$.

The next proposition is an immediate consequence of \cite[Theorem 1]{Go} and will
simplify many of the arguments.

\begin{prop} \label{Nikolaz}
There exists an adapted Riemannian metric $\|\cdot\|_0$, equivalent
to the original one and $\lambda\in(0,1)$ such that
\begin{equation} \label{domino}
\prod_{j=0}^n \|Df| E^{cs,i}(f^j(x))\|_0 \cdot\|Df^{-1}|E^{cu,i+1}(f^{j}(x))\|_0< \lambda^n
\end{equation}
for  all $n\geq 0$, all $x\in\Lambda$, and all $i\in\{0,...,k\}$.
\end{prop}

\begin{proof}
Fix $i\in\{ 0,...,k\}$.  By \cite{Go} there exist an adapted Riemannian metric
equivalent to the original one and $\lambda_i\in (0,1)$ such that
$\|Df|E^{cs,i}(x)\|_0\cdot \|Df^{-1}|E^{cu,i+1}(f(x))\|_0\leq \lambda_i$. Since
the splitting is invariant we conclude that
$$
\prod_{j=0}^n
\|Df| E^{cs,i}(f^j(x))\|_0 \cdot\|Df^{-1}|E^{cu,i+1}(f^{j}(x))\|_0< \lambda_i^n.
$$
Letting $\lambda=\max\{\lambda_0,...,\lambda_k\}$ we prove the lemma.
\end{proof}

Throughout the rest of the paper we assume that the Riemannian metric is an adapted metric.

For each $i=0, \dots, k$,
consider the dominated splitting $E^{cs,i}\oplus E^{cu,i+1}$. 
If $V(\Lambda)$  is a small neighborhood
of $\Lambda$ then   the $f$-invariant set
\begin{equation}\label{e.LV}
\Lambda_V=\bigcap_{n\in\mathbb{Z}}f^n(\overline{V(\Lambda)})
\end{equation}
has a dominated splitting that extends the splitting on $\Lambda$, see for instance~\cite[App.~B]{BDV}. We also denote these extensions by
$ E^{cs,i}$ and $E^{cu,i+1}$. Moreover, these splittings 
can be continuously extended to $V(\Lambda)$. These extensions 
are ``nearly" invariant under $f$. That is, there are sufficiently small cone fields about the extended splitting that are invariant. We denote these extensions by $\widetilde E^{cs,i}$ and  $\widetilde  E^{cu,i+1}$ and 
the  small cone fields by
$\cC(\widetilde  E^{cs,i})$ and $\cC(\widetilde  E^{cu,i+1})$.


\subsection{Fake center manifolds} \label{ss.centers}
Much of this section follows Section 3 of~\cite{BW}.  The next proposition is similar to Proposition 3.1 in~\cite{BW}.  

\begin{prop}\label{p.fake}
Let $f:M\rightarrow M$ be a $C^1$ diffeomorphism and $\Lambda$ a
compact 
$f$-invariant set with a partially hyperbolic splitting, 
$$T_{\Lambda}M=E^s\oplus E_{1}\oplus \cdots \oplus E_{k}\oplus E^u.$$  
Let $E^{cs,i}$ and $E^{cu,i}$ be as in equation \eqref{e.pp} and
consider their extensions 
$\tilde E^{cs,i}$ and $\tilde E^{cu,i}$ to a small neighborhood
of $\Lambda$.

Then for any $\epsilon>0$ there exist constants $R>r>r_1>0$ such that, for every $p\in \Lambda$, the neighborhood $B(p,r)$ is foliated by foliations $\widehat{W}^u(p)$, $\widehat{W}^s(p)$, 
$\widehat{W}^{cs,i}(p)$, and $\widehat{W}^{cu, i}(p)$, $i\in\{1,..., k\}$,  such that for each $\beta\in \{u,s, (cs, i), (cu,i)\}$
the following properties hold:
\begin{enumerate}
\item[(i)] {\em{Almost tangency of the invariant distributions.}}  For each $q\in B(p,r)$, the leaf $\widehat{W}^{\beta}_p (q)$ is $C^1$, and the tangent space $T_q\widehat{W}^{\beta}_p(q)$ lies in a cone of radius $\epsilon$ about $\widetilde{E}^{\beta}(q)$.
\item[(ii)] {\em{Coherence.}}  $\widehat{W}^s_p$ subfoliates $\widehat{W}^{cs,i}_p$  and $\widehat{W}^u_p$ subfoliates $\widehat{W}^{cu, i}_p$ for each $i\in\{1,..., k\}$.
\item[(iii)] {\em{Local invariance.}}  For each $q\in B(p, r_1)$ we have 
$$f(\widehat{W}^{\beta}_p(q,r_1))\subset \widehat{W}^{\beta}_{f(p)}(f(q))\textrm{ and }
f^{-1}(\widehat{W}^{\beta}_p(q,r_1))\subset \widehat{W}^{\beta}_{f^{-1}(p)}(f^{-1}(q)),$$
here $\widehat{W}^{\beta}_p(q,r_1)$ is the connected component of
$\widehat{W}^{\beta}_p(q)\cap B(q,r_1)$ containing $q$.
\item[(iv)] {\em{Uniquencess.}}  $\widehat{W}^s_p(p)=W^s(p,r)$ and $\widehat{W}^u_p(p)=W^u(p,r)$.
\end{enumerate}
\end{prop}

By choosing the neighborhood $V(\Lambda)$
of $\Lambda$ above sufficiently small we have that $\Lambda_V$ also satisfies the hypotheses in Proposition~\ref{p.fake}  and hence the points in $\Lambda_V$ have fake foliations as in the proposition.

\begin{rem}
\label{r.fake}
{\em{ For $\epsilon$ sufficiently small, the transversality of the invariant bundles for $\Lambda_V$ implies that, for all $p\in \Lambda_V$ 
and every $x$ and $y$ sufficiently close to $\Lambda_V$, then
$\widehat{W}^{cs,i}_p(x)\cap\widehat{W}^{cu, i+1}_p(y)$ consists of a single point for all $i\in\{0,..., k\}$. Here 
$\widehat{W}^s_p(x)=\widehat{W}^{cs,0}(x)$ and $\widehat{W}^u_p(y)=\widehat{W}^{cu, k+1}_p(y)$.}}
\end{rem}

The proof of Proposition~\ref{p.fake} is very similar to the one of Proposition 3.1 in~\cite{BW}, inspired by Theorem 5.5 in~\cite{HPS}. 
So it is omitted. In~\cite{HPS} the result is that the leaves would be tangent at $p$ to the invariant bundles, but the leaves do not form necessirely a foliation.  
In the proposition above, the fake foliations are not tangent to
the initial bundles but they
 stay within thin cone fields. The main advantage is that these fake
leaves foliate local neighborhoods. As in \cite{HPS} to get the fake foliations one use a graph transform.

\subsection{Central curves}

Throughout the rest of the paper we fix $\rho>0$ such that for all $i\in\{0,...,k\}$ and $x\in V(\Lambda)$ there exists a curve $\gamma_i(x)$ 
centered at $x$
of radius $\rho$ tangent to the bundle $E_{i}$.

The next lemma is a higher dimensional version of \cite[Lemma 2.2]{PaVi}.
Since  the proof is analogous to the one there we omit it.

\begin{lem}\label{l.sehacenestables}
Let $\Lambda$ be and $f$-invariant set as in 
Theorem~\ref{t.mainresultado}.
For any sufficiently small $\rho>0$  and any  $\delta\in(0,\rho)$,
if $x\in\Lambda_V$, $y\in\Gamma_{\delta}(x)$,
and $\gamma$ is a curve with endpoints $x$ and $y$ that
is contained in $B_{\rho}(x)$ and
tangent to $E_i$ for some $i\in\{1,...,k\}$   then
$$
\ell(\gamma)<2\, \delta\textrm{ and }  \gamma\subset\Gamma_{2\,\delta}(x).
$$
\end{lem}

\section{Proof of Theorem~\ref{t.mainresultado}}\label{s.proof}

We now proceed to prove our main theorem.  The idea of the proof is that using the fake foliations we can show that the set $\Gamma_\epsilon(x)$ is 1-dimensional for each point.  Under iteration this set will stay in a 1-dimensional set of finite radius.  Then a folklore fact will show that the set $\Gamma_\epsilon(x)$ has zero topological entropy.  As we can do this uniformly we will know that $\Lambda$ is entropy expansive with for $f$.


\subsection{Hyperbolic-like behavior of fake foliations}\label{ss.hyperbolicfake}

\begin{rem}[Choice of constants]
\label{constants} {\em{ We  fix some constants:
\begin{enumerate}
\item[(a)] Fix $\tau>0$ such that $(1 +\tau)\sqrt{\lambda}<1$, where $\lambda<1$ is the domination
constant in \eqref{domino}.
\item[(b)] Fix $\nu>0$ sufficiently small such that $\bigcup_{x\in\Lambda}B_{5\,\nu}(x)\subset V(\Lambda)$
and such that if $y,y'\in B_{5\,\nu}(x)$ for some $x\in\Lambda$, then  for all $i\in\{0,...,k\}$ it holds
\begin{equation*}
1-\tau<\frac{\|Df^{-1}|_{\widetilde  E^{cu,i}
(y)}\|}{\|Df^{-1}|_{\widetilde  E^{cu,i} (y')}\|}<1+\tau\quad
\mbox{and} \quad
 1-\tau<\frac{\|Df|_{\widetilde  E^{cs,i}
(y)}\|}{\|Df|_{\widetilde  E^{cs,i} (y')}\|}<1+\tau.
\end{equation*}
\end{enumerate}}}
\end{rem}

To obtain the ``1-dimensionality'' of the sets $\Gamma_\delta(x)$ we use hyperbolic times.  
To do so the first step is the following reformulation
of Pliss Lemma  stated for  the sets $\Lambda_V$ satisfying the hypotheses of Theorem~\ref{t.mainresultado}.

\begin{lem}{\cite{Al,Pl1}} \label{Pliss} Let $\lambda>0$ be as in Proposition~\ref{Nikolaz}
and $0<\lambda<\lambda_1<\lambda_2<1$.  Assume that $x\in \Lambda_V$, $i\in\{0,..., k\}$,
and there exists $n\geq 0$
such that
\begin{equation}\label{e.ht}
\prod_{m=0}^n\|Df|_{ E^{cs,i}(f^m(x))}\|\leq \lambda_1^n\, .
\end{equation}
 Then there is  $N=N(\lambda_1,\lambda_2,f)\in \mathbb{N}$ and a constant
$c=c(\lambda_1,\lambda_2,f)>0$ such that for every $n \geq N$  there
exist $\ell\geq c\,n$ and numbers (hyperbolic times) 
$$
0<  n_1<  n_2  < \cdots < n_\ell < n
$$
such that
$$
\prod_{m=n_r}^h \|Df|_{E^{cs,i}(f^m(x))}\|\leq \lambda_2^{h-n_r},
$$
for all $r=1,2,\ldots ,\ell$ and all $h$ with
$n_r\leq h\leq n .$

Similar assertions hold for the map $f^{-1}$ and the bundles
 $E^{cu,i}$.
\end{lem}


The next application of Lemma~\ref{Pliss} provides a lower
bound for the expansion of  
$Df$ along the bundle $E^{cs,i}$
and of
$Df^{-1}$ along the bundle $E^{cu,i}$.
Given a central curve $\gamma_i(x)$ and points
$y,z\in \gamma_i(x)$ we let
  $[y,z]_{\gamma_i(x)}$ be the segment 
in $\gamma_i(x)$ with endpoints $y,z$. 

\begin{lem}\label{l.contraction} Consider a small enough $\delta >0$.
If $x\in\Lambda_V$, $y\in\Gamma_{\delta}(x)$, $y\neq x$, $i\in\{1,...,k\}$.
Suppose that $y \in \gamma_i(x)$.
Then   $[x,y]_{\gamma_i(x)}\subset \Lambda_V$ and if $\lambda_1\in(\lambda, \sqrt{\lambda})$  then there
is $n_0>0$ such that
$$
\prod_{j=0}^{n} \|Df|_{ E^{cs,i} (f^{j}(f^{-n}(y')))}\| >
\lambda_1^n
\quad
 \textrm{and}
\quad 
\prod_{j=0}^{n} \|Df^{-1}|_{E^{cu,i} (f^{j}(y'))}\| >
\lambda_1^n
$$
for all $y'\in  [x,y]_{\gamma_i(x)}$ and $n>n_0$.
\end{lem}

This lemma
 means that, for $\delta>0$ sufficiently small, 
if $[x,y]_{\gamma_i(x)}\subset \Gamma_{\delta}(x)$ and $y\neq x$ then for
 all point $y'\in[x,y]_{\gamma_i(x)}$ the leaves of the fake foliations $\widehat{W}^{cs,i-1}_x(y')$ and $\widehat{W}^{cu, i+1}_x(y')$ behave  like leaves of the stable and unstable foliations, respectively.  
Thus we  have the following consequence:

\begin{cor}\label{c.todd}
Let $y\in \Gamma_\delta(x) \setminus \{x\}$ such that the central curve
$[x,y]_{\gamma_i(x)}$ is contained in $\Gamma_{\delta}(x)$ for some small $\delta>0$. Then for all
 $y'\in[x,y]_{\gamma_i(x)}$ we have that
$$
\widehat{W}^{cu,i+1}_x(y')\cap \Gamma_\delta (x)=\{y'\}.
\quad
\mbox{and}
\quad
\widehat{W}^{cs,i-1}_x(y')\cap \Gamma_\delta (x)=\{y'\}.
$$ 
\end{cor}

\begin{proof}[Proof of Lemma \ref{l.contraction}]
By Lemma~\ref{l.sehacenestables}, if $\delta$ is sufficiently small then
 the curve
$[x,y]_{\gamma_i(x)}$ is contained in $\Lambda_V$ and thus
the bundles $ E^{cs,i}(y')$ and $ E^{cu,i}(y')$ 
are defined along the orbit of any
$y'\in[x,y]_{\gamma_i(x)}$.

Let  $\lambda_2\in(\lambda_1,1)$ such that $(1+\tau)\,\lambda_2<1$, 
where $\tau$ is  as in
 Remark~\ref{constants}~(b).
Let us prove the first inequality.
 Arguing  by contradiction,
suppose that there exist
infinitely many $m_n\in\mathbb{N}$, $m_n\to \infty$, 
and $y_n\in [x,y]_{\gamma_i(x)}$
such that 
$$
\prod_{j=0}^{m_n} \|Df|_{ E^{cs,i} (f^{j}(f^{-m_n} (y_n)))}\| \le 
\lambda_1^n.
$$
Writing 
$w_n=f^{-m_n}(y_n)$, 
by
Lemma \ref{Pliss} 
we have
$$
\prod_{j=0}^{m_n} \|Df |_{E^{cs,i}(f^{j}(w_n))}\| \leq \lambda_2^{m_n}.
$$
By Lemma~\ref{l.sehacenestables}, for all $j\geq 0$, the curve $f^{-j}([x,y]_{\gamma_i(x)})$
stays $2\,\epsilon$-close to $f^{-j}(x)$,
and thus $2\,\epsilon$-close to $f^{-j}(w_n)$.  It follows from
Remark~\ref{constants}~(b)
  that, for all $y'\in f^{-m_n} ([x,y]_{\gamma_i(x)})$, one has
$$
\prod_{j=0}^{m_n} \|Df |_{E^{cs,i}(f^{j}(y'))}\| \leq 
\big((1+\tau)\,
\lambda_2\big)^{m_n}.
$$
Since $[x,y]_{\gamma_i(x)}=f^{m_n}\big([f^{-m_n}(x), f^{-m_n}(y)]_{
\gamma_i(  f^{-m_n}(x) ) } \big)$
we have that
$$
\ell \big([x,y]_{\gamma_i(x)} \big)
\leq \big( (1 +\tau)\,\lambda_2 \big)^{m_n}\,\,
\ell\big( [f^{-m_n}(x),f^{-m_n}(y)]_{\gamma_i(  f^{-m_n}(x) ) }\big).
$$
By Lemma~\ref{l.sehacenestables}, if $\epsilon$ is small enough,
 $\ell\big([f^{-m_n}(x),f^{-m_n}(y)]_{\gamma_i(  f^{-m_n}(x) ) }\big)$ is bounded by
$2\, \delta$. Thus
letting
$m_n\to+\infty$
  we get that $x=y$, a contradiction.
  
 To get the other product we simply look to the
bundle $E^{cu,i}$ and the map $f^{-1}$ and repeat the above argument. 
\end{proof}

\subsection{End of the proof of Theorem~\ref{t.mainresultado}}
\label{ss.end}
The main step of the   proof of Theorem~\ref{t.mainresultado} is the following
result.

\begin{prop}\label{p.bowenballs}
For $x\in\Lambda$ and $\delta>0$ sufficiently small the set $\Gamma_\delta(x)$ is either $\{x\}$ or is contained in a curve $\gamma_i(x)$ for some $i\in\{1,..., k\}$.
\end{prop}

We postpone the proof of this proposition
and prove the theorem.

\medskip

\noindent{\bf{Proof of Theorem~\ref{t.mainresultado}.}}
Let $f$ and $\Lambda$ satisfy the hypothesis of Theorem~\ref{t.mainresultado}.  Then from Proposition~\ref{p.bowenballs} we know that for $\delta>0$ sufficiently small the set $\Gamma_{\delta}(x)$ is a single point or contained in a central curve $\gamma_i(x)$.   It is  a folklore fact 
that the entropy of 1-dimensional curves of bounded length is zero (for a proof see for instance \cite{BFSV}). 
This implies that $\tilde h(f, \Gamma_\epsilon(x))=0$ for all $x$ and 
every small $\epsilon$.
 Hence, the set $\Lambda$ is entropy expansive for $f$. \hfill
$\Box$

\bigskip

\noindent{\bf{Proof of Proposition~\ref{p.bowenballs}.}}
  Fix $x\in \Lambda$ and assume that 
there is 
 $y\in\Gamma_\delta(x)\setminus \{x\}$. We start with the following lemma.

\begin{lem}\label{l.1}
For every small enough $\delta>0$,
$$
\Gamma_\delta(x) \subset
\widehat W^{cs,k}_x(x) \cap \widehat W^{cu,0}_x(x).
$$
\end{lem}

\begin{proof}
We see that $\Gamma_\delta(x) \subset
\widehat W^{cs,k}_x(x)$, the inclusion 
$\Gamma_\delta(x) \subset
\widehat W^{cu,0}_x(x)$ follows similarly. 
Take $x\in \Gamma_\delta(x)$.
If  $y\not\in \widehat{W}^{cs,k}_x(x)$,
then
as $\widetilde{E}^u$ is uniformly expanding, after forward iterations
the orbit of  $y$ will scape from 
$\widehat{W}^{cs,k}_x(x)$ and thus from
 $x$,  contradicting that $y\in \Gamma_\delta(x)$. This ends the proof
of the lemma.
\end{proof}

Given $j\in \{1,\dots,k\}$, 
using Proposition~\ref{p.fake}, 
we consider small $r$ and the
submanifold
\begin{equation}\label{e.ffake}
\widetilde W^{cs,j}(x)=\bigcup_{z\in \gamma_j(x)}\,
\widehat W^{cs,j-1}_x(z,r).
\end{equation}
This submanifold has dimension $s+j$ and is transverse
to $\widehat W^{cu,j+1}_x(z)$ for all $z$ close to $x$.
Note that 
$\widetilde W^{cs,1}(x)$ is foliated by stable manifolds
(recall that $\widehat W^{cs,0}_x(z)\subset W^{s}(z)$).

For every $j\in \{1, \dots, k\}$ and every
$y\in \Gamma_{\delta}(x) \cap \widehat W^{cs,j}_x(x)$
we associate
a pair of points 
 $\widetilde y_{j}\in
\widetilde W^{cs,j}(x)$ and $\overline y_j\in \gamma_j(x)$ 
defined as follows, see Figure~\ref{f.zz},

\begin{equation}\label{e.pqp}
\widetilde y_{j} \eqdef  W^{cu,j+1}_x (y) \cap \widetilde W^{cs,j}(x),
\quad \mbox{where} \quad
\widetilde y_j = \widehat W^{cs,j-1}_x (\overline y_j) \cap \gamma_j(x).
\end{equation}

\begin{figure}[htb]

\psfrag{y}{$y$} 
\psfrag{x}{$x$} 
\psfrag{by}{$\overline y_j$} 
\psfrag{ty}{$\widetilde y_j$} 
\psfrag{tw}{$\widetilde W^{cs,j}(x)$}
\psfrag{hw}{$\widehat W^{cs,j-1}_x (\overline y_j)$}
\psfrag{wy}{$\widehat W^{cu,j+1}_x (y)$}
\psfrag{g}{$\gamma_j(x)$}

   \includegraphics[width=8cm]{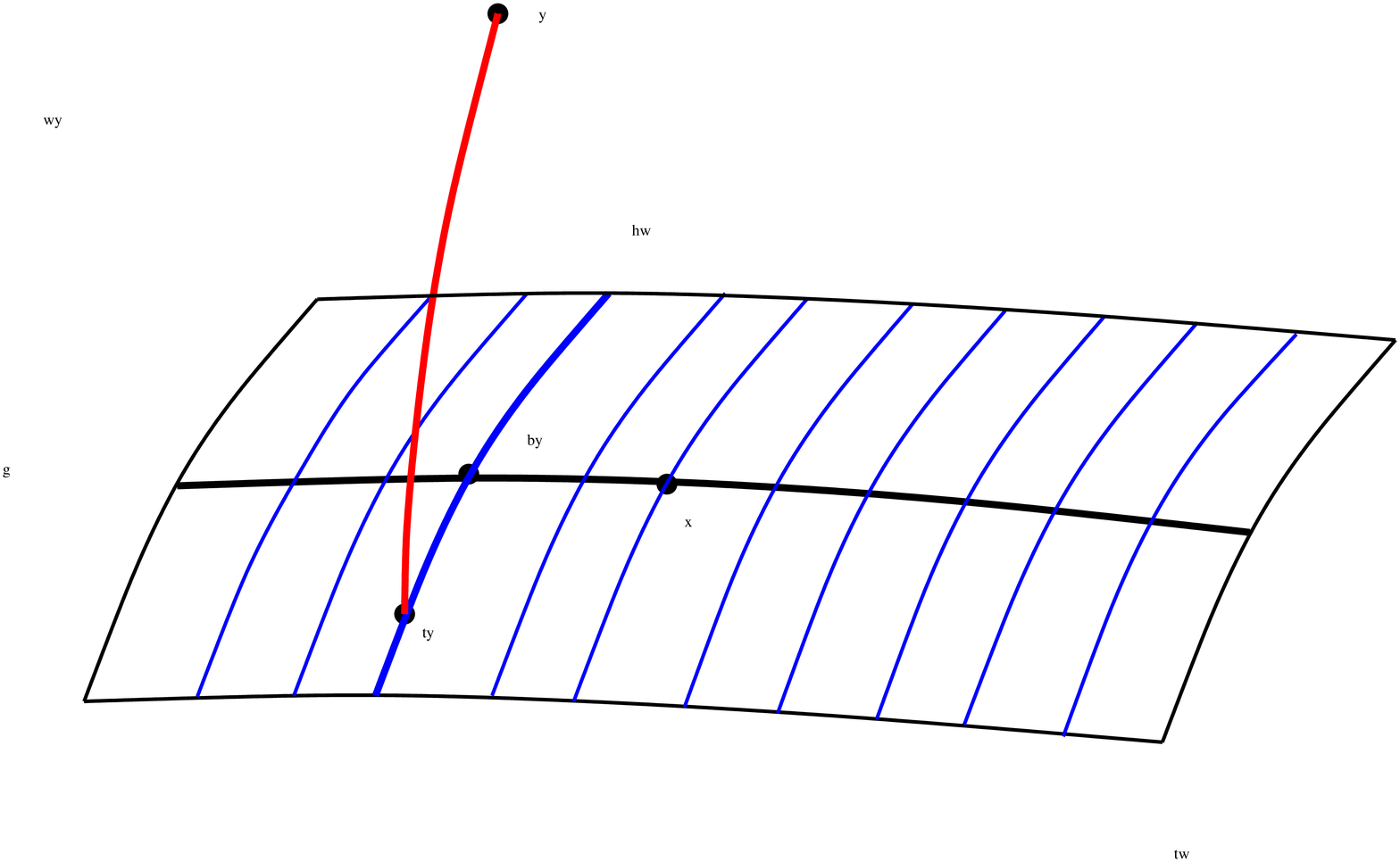}\hspace{1cm}
    \caption{The points
$\widetilde y_j$ and 
 $\overline y_j$.}
  \label{f.zz}
\end{figure}

\begin{cla}\label{c.1}
Given small $\delta>0$ and 
 $y\in \Gamma_{\delta}(x) \cap \widehat W^{cs,j}_x(x)$ then
\begin{enumerate}
 \item[a)] either $\overline y_j\ne x$,
\item[b)] $\overline y_j = x$ and  $\widetilde y_j\ne x$.
\end{enumerate}
\end{cla}

\begin{proof}
It is enough  to see that the case $\overline y_j=\widetilde y_j=x$ can not occur. 
If so, by Remark~\ref{r.fake}, $y\in \widehat W^{cu,j+1}_x(x) \cap \widehat W^{cs,j}_x(x)=\{x\}$,
which is a contradiction.
\end{proof}

The  next two lemmas follow straighforwardly 
from the fact that the angles 
between unitary vectors in the cone fields
$\cC (E^{cs,j})$ and 
 $\cC(E^{cu,j+1})$ are uniformly bounded away from zero.

\begin{lem}\label{l.fkanglesa}
 There is $\kappa>0$ such that for every $j\in \{1,\dots ,k\}$ and
every $\delta>0$ small enough the following property holds:

For every $x\in \Lambda$, every $y\in B_\delta(x)$, 
every local submanifolds $N(x)$ of dimension  $s+j$ 
tangent to the conefield $\cC (E^{cs,j})$ containing $x$
and $M(y)$ of dimension $(k-j)+u$
tangent to the conefield  $\cC(E^{cu,j+1})$ containing $y$
one has that
$N(x) \cap M(y)$ is contained $B_{\kappa\,\delta}(x)$. 
\end{lem}

\begin{lem}\label{l.fkanglesb}
 There is $\kappa>0$ such that for every $j\in \{1,\dots ,k\}$ and
every $\delta>0$ small enough the following property holds:

Take any  $x\in \Lambda$
and the local manifold $\widetilde W^{cs,j}(x)$ in \eqref{e.ffake}.
For every $y\in B_\delta(x)\cap \widetilde W^{cs,j}(x)$ 
one has that
$\gamma_j(x) \cap 
\widehat W^{cs,j-1}_x (y)$ is contained in $B_{\kappa\,\delta}(x)$. 
\end{lem}

The next lemma is similar to the above, but is concerned with what happens inside a submanifold tangent to a conefield.

\begin{lem}\label{l.3}
There is $\widehat \kappa>0$ such 
that if 
 $y\in \Gamma_{\delta}(x) \cap \widehat W^{cs,j}_x(x)$ then if
$\delta>0$ is small enough then 
$\overline y_j,\widetilde y_j\in 
\Gamma_{\widehat \kappa \,\delta}(x)$.
\end{lem}

\begin{proof}
For simplicity let us omit the subscrit $j$ and just write
$\widetilde y$ and $\overline y$.
Take $r_1$ as in Proposition~\ref{p.fake}.

\begin{cla}\label{c.sacocheio} 
If  $\delta>0$ is small enough then 
for all $i\ge 0$ it holds
\[
f^i(\widetilde y) \in \widehat W^{cu,j+1}_{f^i(x)} (f^i(y), r_1)\quad
\mbox{and}
\quad
f^i(\widetilde y) \in \widehat W^{cs,j-1}_{f^i(x)} (f^i(\overline y), r_1)
\]
and $f^i(\overline y)$ is in a central curve $\gamma_j(x,r_1)$ 
centered at $x$ of radious $r_1$ 
tangent to $E_j$ and
containing $x$.
\end{cla}

\begin{proof}
The proof goes by induction on $i$. 
For $i=0$, by 
Lemma~\ref{l.fkanglesa}, 
we have that
$\dist (y, \widetilde y)< \kappa\, \delta< r_1$ and thus
$\dist (\widetilde y,x)< (\kappa +1)\, \varepsilon$.
Hence, for small $\delta$, Lemma~\ref{l.fkanglesb}
implies that $\dist (\widetilde y, \overline y)< \kappa\, (\kappa+1) \, \delta< r_1$. 
By construction the point $\overline y$ is in a central curve 
$\gamma_j(x,r_1)$.
This ends the first inductive step.

Assume that the induction hypothesis holds for all $i=0, \dots, m$. 
By Proposition~\ref{p.fake},
this implies
\[
\begin{split}
f^{m+1} (\widetilde y) &\in 
f \big( \widehat W^{cs,j-1}_{f^m(x)} (f^m(\overline y),r_1)\big)
\subset 
\widehat W^{cs,j-1}_{f^{m+1}(x)} (f^{m+1}(\overline y)), \quad \mbox{and}\\
f^{m+1} (\widetilde y) &\in 
f \big( \widehat W^{cu,j+1}_{f^m(x)} (f^m(y),r_1)\big)
\subset 
\widehat W^{cu,j+1}_{f^{m+1}(x)} (f^{m+1}(y)).
\end{split}
\]
As $f^m(\overline y)\in \gamma_j(f^m(x),r_1)$ we get that
$f^{m+1}(\overline y)\in \gamma_j(f^{m+1}(x))$ and thus,
recalling the definition of 
$\widetilde W^{cs,j}(f^{m+1}(x))$ in \eqref{e.ffake},  
$$
f^{m+1} (\overline y)= \widehat W^{cu,j+1}_{f^{m+1}(x)} (f^{m+1}(y))
\cap \widetilde W^{cs,j}(f^{m+1}(x)).
$$
As $\dist (f^{m+1}(y), f^m(x))<\delta$, we can apply the arguments in 
the step $i=0$ to obtain the claim.
\end{proof}

By Lemma~\ref{l.fkanglesa}, we have
$\dist (f^{m+1}(y), f^{m+1}(\widetilde y))< \kappa\, \delta$ and
hence $\dist (f^{m+1}(x), f^{m+1}(\widetilde y))< (\kappa +1) \, \delta$.
Lemma~\ref{l.fkanglesb} now implies that
$\dist (f^{m+1}(\overline y), f^{m+1}(\widetilde y))< \kappa\, 
(\kappa+1)\,
\delta$. Hence 
$$
\dist (f^{m+1}(x), f^{m+1}(\overline y))< 
2\, (\kappa +1)^2 \, \delta.
$$
Taking $\widehat \kappa= 2\, (\kappa+1)^2$ we 
end the proof of Lemma~\ref{l.3}.
\end{proof}

\begin{lem} \label{l.nense}
For every   $\delta>0$ small such that
there is $y\in \Gamma_\delta(x)\setminus \{x\}$
there is $j_0\in \{1,\dots,k\}$ 
such that 
$$
\gamma_{j_0}(x) \cap \big( \Gamma_{\kappa'\, \delta}(x)\setminus
\{x\} \big)\ne \emptyset
, \quad \mbox{where  $\kappa'=\widehat \kappa^{\,k}$.}
$$ 
\end{lem}

\begin{proof}
The next claim is needed in the proof of the lemma.
\begin{cla}\label{c.mengo}
Let $j\in \{1, \dots, k\}$ and
$y\in \widehat W^{cs,j}_x(x) \cap \Gamma_{\delta}(x)$, $y\ne x$.
Then there are two possibilities:
\begin{enumerate}
 \item[a)] either $\gamma_j(x)$ contains at least two points
of $\Gamma_{\widehat \kappa \, \delta}(x)$,
\item[b)] or there is $\widehat y \in
\widehat W^{cs,j-1}_x(x) \cap \Gamma_{\widehat \kappa\, \delta}(x)$.
\end{enumerate}
\end{cla}

\begin{proof}
If  $y \in  \gamma_j(x)$ we are done. 
Otherwise consider the points $\overline y_j$ and
$\widetilde y_j$ defined in equation \eqref{e.pqp}.
If $\overline y_j\ne x$, by Lemma~\ref{l.3}, 
$\overline  y_j\in \gamma_j(x) \cap \Gamma_{\widehat \kappa\, \delta}(x)$ and
we are also done.
Otherwise,
$\overline y_j= x$ and
 we take the point $\widetilde y_j\in \big( 
\Gamma_{\widehat \kappa \, \delta}(x)
\setminus \{x\}\big)$, recall Claim~\ref{c.1}. This point belongs to
$\widehat W^{cs,j-1}_x(\overline y_j)=
\widehat W^{cs,j-1}_x(x)$. Taking $\widehat y=\widetilde y_j$ one proves the claim.
\end{proof}

We are now ready to end the proof of Lemma~\ref{l.nense}.
By Lemma~\ref{l.1} we have $y\in \widehat W^{cs,k}_x(x)$.
Let $y\eqdef y_k$. Recursively, using the notation in
Claim~\ref{c.mengo}, for $j=1,\dots, k$, we define the points
 $y_{j-1}\eqdef \widetilde{y_j}$ where
\begin{equation}\label{e.docaral}
 y_{j-1}= \widetilde{y_j}\in 
\widehat W^{cs,j-1}_x(x)
\cap
\Gamma_{\widehat \kappa^{k-j}\, \delta}(x), \quad y_{j-1}\ne x.
\end{equation}
Since we are assuming that $\gamma_j(x)=\{x\}$ for all $j=1,\dots,k$,
by Claim~\ref{c.mengo}, the points $y_j$ are well defined. 
By construction, we have $y_0\in W^s(x)$, which is a contradiction.
This proves the lemma.
\end{proof}

\begin{sch}\label{sc.elon}{\em{
The arguments in the proof of Lemma~\ref{l.nense} implies the following.
Assume that
$\gamma_j(x)=\{x\}$ for $j=\iota+1, \dots, k$.
Then to any point $y\in \Gamma_{\delta}(x)$
we can  associate points
$y_k=y,\, y_{k-1},\dots y_{\iota+1}\in 
\big(\Gamma_{\kappa'\, \delta}(x)\setminus
\{x\}\big)$ such that
$$
y_{j}\in 
\widehat W^{cs,j}_x(x)
\cap
\Gamma_{\kappa'\, \delta}(x).
$$ 
}}
\end{sch}

\medskip

\noindent{{\emph{End of the proof of Proposition~\ref{p.bowenballs}.} 
In view of Lemma~\ref{l.nense},  there is a largest $j\in\{1,\dots,k\}$,
that we denote by $\iota$, such that 
$\gamma_\iota(x)\cap \Gamma_{\kappa'\,\delta} (x)$ contains 
at least two points.
Then $\gamma_\iota (x)\cap \Gamma_{\kappa'\,\delta} (x)=[a,b]$,
with $a\ne b$.

Since $\gamma_j(x)=\{x\}$, 
 for $j=\iota+1, \dots, k$ we can consider the points
$y_k=y, \, y_{k-1},\dots y_{\iota+1}\in 
\big( \Gamma_{\kappa'\, \delta}(x)\setminus
\{x\} \big)$ satisfying
Scholium~\ref{sc.elon}.

For $z=y_{\iota+1}$  we let
$\widetilde z_{\iota+1},
\overline z_{\iota+1}\in \Gamma_{\widehat \kappa\,\kappa'\,\delta}
(x)$
as in Equation \eqref{e.pqp}.  
Define
\begin{equation}
\widetilde y^{\,\,*} =
\widetilde z_{\iota +1}, \quad 
 \overline y^{\,\,*} =
\overline z_{\iota +1}.
\end{equation}
By Claim~\ref{c.1} and by construction,
there are two possibilities:
\begin{enumerate}
 \item 
$\widetilde y^{\,\,*}\in W^{cs, \iota-1}_x 
( \overline y^{\,\,*})$,
where  
$\overline y^{\,\,*}\in [a,b]$ and $\widetilde y^{\,\,*}\ne
\overline y^{\,\,*}$, \, or
\item 
$y_{\iota +1} \in  W^{cu, \iota +1}_x 
( \overline y^{\,\,*})$ and 
$\overline y^{\,\,*}\in [a,b]$.
\end{enumerate}
Note that  $\overline y^{\,\,*}\ne x$.
As $[a,b]$ is non-trivial and
$\overline y^{\,\,*} \in [a,b] \subset \Gamma_{\widehat \kappa \, 
\kappa'\, \delta}(x)$,
Corollary~\ref{c.todd} implies that, if case (1) holds then
$\widetilde y^{\,\,*}\notin \Gamma_{\widehat \kappa \, 
\kappa' \, \delta}(x)$.  Similarly, if  case (2) holds,
$y_{\iota +1}\not\in 
\Gamma_{\widehat \kappa \,  \kappa'
\, \delta}(x)$. In both cases we get a contradiction, ending the proof of the proposition. \hfill \quad

\end{document}